\documentclass[12pt,a4paper,reqno]{amsart}
\usepackage{amssymb}
\usepackage{amscd}
\theoremstyle{plain}

\newtheorem{theorem}[subsection]{Theorem}
\newtheorem{proposition}[subsection]{Proposition}
\newtheorem{lemma}[subsection]{Lemma}
\newtheorem{corollary}[subsection]{Corollary}

\newtheorem{remark}[subsection]{Remark}

\newcommand\R{\mathbb{R}}
\newcommand\T{\mathbb{T}}

\newcommand\N{\mathbb{N}}

\newcommand\eps{{\varepsilon}}

\begin{document}

\title{Asymptotic decay for a one-dimensional nonlinear wave equation}

\author{Hans Lindblad}
\address{Department of Mathematics, UCSD, San Diego CA 92013-0112}
\email{lindblad@math.ucsd.edu}

\author{Terence Tao}
\address{Department of Mathematics, UCLA, Los Angeles CA 90095-1555}
\email{tao@math.ucla.edu}
\thanks{HL is supported by NSF grant DMS-0801120 and TT is supported by NSF Research Award DMS-0649473, the NSF Waterman award and a grant from the MacArthur Foundation.  We thank David Str\"utt, Jason Murphy, Qing Tian Zhang and the anonymous referee for corrections.}

\begin{abstract}  We consider the asymptotic behaviour of finite energy solutions to the one-dimensional defocusing nonlinear wave equation $-u_{tt} + u_{xx} = |u|^{p-1} u$, where $p > 1$.  Standard energy methods guarantee global existence, but do not directly say much about the behaviour of $u(t)$ as $t \to \infty$.  Note that in contrast to higher-dimensional settings, solutions to the linear equation $-u_{tt} + u_{xx} = 0$ do not exhibit decay, thus apparently ruling out perturbative methods for understanding such solutions.  Nevertheless, we will show that solutions for the nonlinear equation behave differently from the linear equation, and more specifically that we have the average $L^\infty$ decay $\lim_{T \to +\infty} \frac{1}{T} \int_0^T \|u(t)\|_{L^\infty_x(\R)}\ dt = 0$, in sharp contrast to the linear case.  An unusual ingredient in our arguments is the classical Radamacher differentiation theorem that asserts that Lipschitz functions are almost everywhere differentiable.
\end{abstract}

\maketitle

\section{Introduction}

Fix $p>1$.  We consider solutions $u: \R \times \R \to \R$ to the one-dimensional defocusing nonlinear wave equation
\begin{equation}\label{nlw}
-u_{tt} + u_{xx} = |u|^{p-1} u
\end{equation}

with the finite energy initial condition
$$\| u(0) \|_{H^1_x(\R)} + \| u_t(0) \|_{L^2_x(\R)} < \infty.$$
Standard energy methods (using the Sobolev embedding $H^1_x \subset L^\infty_x$) show that the initial value problem is locally well-posed in this energy class.  Furthermore, by using the conservation of energy\footnote{In order to justify energy conservation for solutions which are in the energy class, one can use standard local well-posedness theory to approximate such solutions by classical (i.e. smooth and compactly supported) solutions (regularising the nonlinearity $|u|^{p-1}u$ if necessary), derive energy conservation for the classical solutions, and then take strong limits.  We omit the standard details.  More generally, we shall perform manipulations such as integration by parts on finite energy solutions as if they were classical without any further comment.}
\begin{equation}\label{energy}
 E[u] = E[u(t)] := \int_\R \T_{00}(t,x)\ dx
 \end{equation}
where $\T_{00}$ is the \emph{energy density}
$$ \T_{00} := \frac{1}{2} u_t^2 + \frac{1}{2} u_x^2 + \frac{1}{p+1} |u|^{p+1}$$
it is easy to show that the $H^1_x \times L^2_x$ norm of $u(t)$ does not blow up in finite time, and that the solution to \eqref{nlw} can be continued globally in time.

In this paper we study the asymptotic behaviour of finite energy solutions $u$ to \eqref{nlw} as $t \to \pm \infty$.  Of course, from the conservation of energy \eqref{energy} we know that $u(t)$ stays bounded in $\dot H^1_x(\R) \cap L^{p+1}_x(\R)$, and thus (by the Gagliardo-Nirenberg inequality) bounded in $L^\infty_x(\R)$ for all time, but this does not settle the question of whether $\|u(t)\|_{L^\infty_x(\R)}$ exhibits any \emph{decay} as $t \to \pm \infty$.

For the linear equation $-u_{tt} + u_{xx} = 0$, the solutions are of course travelling waves $u(t,x) = f(x+t) + g(x-t)$, which do not decay along light rays $x = x_0 \pm t$. In particular, for any non-trivial linear solution, $\|u(t)\|_{L^\infty_x(\R)}$ stays bounded away from zero.  It is thus natural to ask whether the same behaviour occurs for solutions to the nonlinear equation \eqref{nlw}.  However, an easy energy argument shows that the behaviour must be slightly different.  Indeed, if we introduce the \emph{momentum density} (or \emph{energy current})
$$ \T_{01} = \T_{10} := u_t u_x$$
and the \emph{momentum current}
$$ \T_{11} := \frac{1}{2} u_t^2 + \frac{1}{2} u_x^2 - \frac{1}{p+1} |u|^{p+1}$$
we observe the conservation laws
\begin{align}
\partial_t \T_{00} &= \partial_x \T_{01}\label{enc}\\
\partial_t \T_{01} &= \partial_x \T_{11}. \label{momc}
\end{align}
From \eqref{enc} and the fundamental theorem of calculus we have
$$ \partial_t \int_{x < x_0+t} \T_{00}(t,x)\ dx = \T_{00}(t,x_0+t) + \T_{01}(t,x_0+t)$$
for all $x_0, t \in \R$.
On the other hand, from the non-negativity of $\T_{00}$ we clearly have
$$ 0 \leq \int_{x < x_0+t} \T_{00}(t,x)\ dx  \leq E[u].$$
From the fundamental theorem of calculus (and the monotone convergence theorem), we thus obtain
$$ \int_{-\infty}^\infty \T_{00}(t,x_0+t) + \T_{01}(t,x_0+t)\ dt \leq E[u]$$
for all $x_0 \in \R$.
From the pointwise inequality $\T_{00} + \T_{01} \geq \frac{1}{p+1} |u|^{p+1}$ we conclude in particular the nonlinear decay estimate
\begin{equation}\label{light1}
 \int_{-\infty}^\infty |u|^{p+1}(t,x_0+t)\ dt \leq (p+1) E[u]
\end{equation}
for any $x_0 \in \R$.  From reflection symmetry we also have
\begin{equation}\label{light2}
 \int_{-\infty}^\infty |u|^{p+1}(t,x_0-t)\ dt \leq (p+1) E[u]
\end{equation}
for any $x_0 \in \R$.  We thus see that solutions to the nonlinear equation $u$ must decay (on average, at least) along any light ray $x= x_0 \pm t$, in sharp contrast to solutions to the linear equation.  This simple calculation already reveals that the nonlinear equation has somewhat different asymptotic behaviour from the linear equation, and in particular that it is highly unlikely that one can asymptotically analyse the former as a perturbation of the latter.  This is in contrast with the one-dimensional nonlinear Klein-Gordon equation, for which the decay can be leveraged to obtain asymptotic results; see for instance \cite{lind}. Another contrast is with the local theory, which asserts that singularities for the nonlinear wave equation propagate along the same light rays as for the linear one; see \cite{reed}.

The estimates \eqref{light1}, \eqref{light2} imply that finite energy solutions $u$ cannot concentrate on light rays $\{ (t, x_0 \pm t): t \in \R \}$.  However, it is \emph{a priori} conceivable that such solutions might still concentrate on other worldlines $\{ (t,x(t)): t \in \R \}$.  Concentration on spacelike worldlines (in which $|x'(t)| > 1$) are easily ruled out by finite speed of propagation (or by a modification of the arguments used to derive \eqref{light1}, \eqref{light2}), but concentration on timelike worldlines (in which $|x'(t)|<1$) are not so obviously ruled out.  Nevertheless, we are able to rule out this scenario by the following theorem, which is the main result of this paper.

\begin{theorem}[Average $L^\infty_x$ decay]\label{main}  Let $u$ be a finite energy solution to \eqref{nlw}, with an upper bound $E[u] \leq E$ on the energy.  Then
$$ \frac{1}{2T} \int_{t_0-T}^{t_0+T} \|u(t)\|_{L^\infty_x(\R)}\ dt \leq c_{E,p}(T)$$
for all $t_0 \in \R$ and $T > 0$, where $c_{E,p}: \R^+ \to \R^+$ is a function depending only on the energy bound $E$ and the exponent $p$ such that $c_{E,p}(t) \to 0$ as $t \to \infty$.  In particular, we have
$$ \lim_{T \to +\infty} \sup_{t_0 \in \R} \frac{1}{2T} \int_{t_0-T}^{t_0+T} \|u(t)\|_{L^\infty_x(\R)}\ dt  = 0.$$
\end{theorem}

The proof of this theorem will use energy estimates combined with a version of the Rademacher differentiation theorem (or Lebesgue differentiation theorem), that Lipschitz functions are almost everywhere differentiable.  The basic idea is to observe that if $u$ concentrates on a timelike worldline $\{ (t,x(t)): t \in \R \}$, then $x$ should be Lipschitz, and thus mostly differentiable.  This implies that $u$ concentrates on certain parallelograms in spacetime; we will then use energy estimates to rule out such concentration.

In principle, the decaying bound $c_{E,p}(T)$ could be made explicit, but this would require a quantitative version of the Radamacher differentiation theorem.  Such results exist (see \cite{tao:besi} or \cite[Section 2.4]{structure}), but they are fairly weak (involving the inverse tower exponential function $\log_*$). Presumably a more refined argument than the one given in this paper would give better bounds.  For instance, it is plausible to conjecture that $\|u(t)\|_{L^\infty_x(\R)}$ should decay at a polynomial rate in $t$, at least in the perturbative regime when $u$ is small.

We remark that our methods do not seem to give any precise asymptotics for the solution.  Of course Theorem \ref{main} indicates that the solution will not scatter to a linear solution, but it is not clear what the solution scatters to instead, even in the perturbative regime.  It may be that techniques from nonlinear geometric optics could be useful to settle this question, but the extremely weak decay of the solution means that it would be very difficult for these methods to be made rigorous, at least until one can improve the results of Theorem \ref{main} significantly.

\section{Energy estimates}

In this section we derive the basic energy estimates needed to establish Theorem \ref{main}.
Henceforth we fix $p$ and the finite energy solution $u$.  We adopt the notation $X \lesssim Y$ or $X = O(Y)$ to denote the estimate $|X| \leq CY$, where $C$ can depend on $p$ and the energy bound $E$.  Thus from energy conservation we obtain the bounds
\begin{equation}\label{energy-cons}
 \int_\R |u_t|^2(t,x) + |u_x|^2(t,x) + |u|^{p+1}(t,x)\ dx \lesssim 1
\end{equation}
for all $t$.

\begin{lemma}[H\"older continuity]  For all $t,x,t',x' \in \R$ we have the pointwise bound
\begin{equation}\label{ubound}
u(t,x) = O(1)
\end{equation}
and the H\"older continuity property
\begin{equation}\label{uholder}
u(t,x) - u(t',x') = O(|t-t'|^{1/2} + |x-x'|^{1/2}).
\end{equation}
\end{lemma}

\begin{proof}
The bound \eqref{ubound} follows immediately from \eqref{energy-cons} and the Gagliardo-Nirenberg inequality.
Using the bound on $|u_x|^2$ in \eqref{energy-cons} together with the fundamental theorem of calculus and the Cauchy-Schwarz inequality, we also have the spatial H\"older continuity bound
$$ u(t,x) - u(t,x') = O(|x-x'|^{1/2}).$$
Thus to prove \eqref{uholder} it will suffice to show that
\begin{equation}\label{utt}
u(t_1,x_0) - u(t_2,x_0) = O( (t_2-t_1)^{1/2} )
\end{equation}
for all $t_2 > t_1$.  In view of \eqref{ubound} we may also assume $t_2 = t_1+O(1)$.

Fix $t_1,t_2$.  From \eqref{momc} and the fundamental theorem of calculus we have
$$ \partial_t \int_{x < x_0} \T_{01}(t,x)\ dx = \T_{11}(t,x_0)\ ;$$
integrating this in time and using \eqref{energy-cons} we obtain the bounds
$$ \int_{t_1}^{t_2} \T_{11}(t,x_0)\ dt = O(1).$$
Combining this with \eqref{ubound} we conclude
$$ \int_{t_1}^{t_2} u_t(t,x_0)^2\ dt = O(1)$$
and \eqref{utt} follows from the fundamental theorem of calculus and Cauchy-Schwarz.
\end{proof}

Now we prove a more advanced energy estimate.

\begin{proposition}[Nonlinear energy decay in a parallelogram]\label{parallel}  Let $T \geq R \geq 1$, let $x_0,t_0 \in \R$, and let $v \in \R$ be a velocity.  Then we have
\begin{equation}\label{t0t}
\int_{t_0-T}^{t_0+T} \int_{x_0+vt-R}^{x_0+vt+R} |u(t,x)|^{p+1}\ dx dt \lesssim R^{1/2} T^{1/2} + \frac{T}{R}.
\end{equation}
\end{proposition}

\begin{remark} Energy conservation \eqref{energy-cons} only gives the bound of $O(T)$ for this integral, thus this proposition is non-trivial when $T$ is much larger than $R$.  A key point here is that the bounds do not blow up in the neighbourhood of the speed of light $v=1$.  It may be possible to improve the right-hand side of \eqref{t0t}, and to also control other components of the energy, but the above bound will suffice for our purposes.
\end{remark}

\begin{proof}  By translation invariance we can set $x_0=t_0=0$.  By reflection symmetry we may assume that $v \geq 0$.

Let $\chi: \R \to \R$ be a non-negative bump function supported on $[-2,2]$ which equals $1$ on $[-1,1]$, and let $\psi(x) := \int_{y<x} \chi(y)\ dy$ be the antiderivative of $\chi$.  From \eqref{momc} and integration by parts we have
$$ \partial_t \int_\R \psi(\frac{x-vt}{R}) \T_{01}(t,x)\ dx = - \frac{1}{R} \int_\R \chi(\frac{x-vt}{R}) (\T_{11}(t,x) + v \T_{01}(t,x))\ dx;$$
integrating this against $\chi(t/T)$ using \eqref{energy-cons} we conclude that
\begin{equation}\label{t1b}
 \int_\R \int_\R \chi(\frac{t}{T}) \chi(\frac{x-vt}{R}) (\T_{11}(t,x) + v \T_{01}(t,x))\ dx dt = O( R ).
\end{equation}
A similar argument using \eqref{enc} instead of \eqref{momc} yields
\begin{equation}\label{t2b}
 \int_\R \int_\R \chi(\frac{t}{T}) \chi(\frac{x-vt}{R}) (\T_{01}(t,x) + v \T_{00}(t,x))\ dx dt = O( R ).
\end{equation}
On the other hand, if we define the nonlinear null form
$$ Q :=(-\partial_{tt} + \partial_{xx}) u^2 = - 2 u_t^2 + 2 u_x^2 + 2 |u|^{p+1}$$
then from integration by parts and \eqref{ubound} we have
\begin{equation}\label{t3b}
\begin{split}
\left|\int_\R \int_\R \chi(\frac{t}{T}) \chi(\frac{x-vt}{R}) Q(t,x)\ dx dt\right|
&= \left|\int_\R \int_\R u^2(t,x) (-\partial_{tt} + \partial_{xx})(\chi(\frac{t}{T}) \chi(\frac{x-vt}{R})) \ dx dt\right| \\
&\lesssim \int_{-2T}^{2T} \int_{v-2R}^{v+2R} \frac{1}{T^2} + \frac{1}{R^2}\ dx dt \\
&\lesssim \frac{R}{T} + \frac{T}{R} \\
&\lesssim \frac{T}{R}.
\end{split}
\end{equation}

Let us compare $|u|^{p+1}$ against the quantities
\begin{align*}
\T_{11} + v \T_{01} &= \frac{1}{2} u_t^2 + v u_t u_x + \frac{1}{2} u_x^2 - \frac{1}{p+1} |u|^{p+1} \\
\T_{01} + v \T_{00} &= \frac{v}{2} u_t^2 + u_t u_x + \frac{v}{2} u_x^2 + \frac{v}{p+1} |u|^{p+1} \\
Q &= - 2 u_t^2 + 2 u_x^2 + 2 |u|^{p+1}.
\end{align*}

We divide into three cases.

{\bf Case 1: (Spacelike case) $v \geq 1$.}  In this case, we can verify the pointwise bound
$$ \frac{1}{p+1} |u|^{p+1} \leq \T_{01} + v \T_{00} $$
and so \eqref{t0t} follows immediately from \eqref{t2b} (note that $R = O( R^{1/2} T^{1/2} )$).

{\bf Case 2: (Lightlike case) $1 - \frac{R^{1/2}}{2T^{1/2}} < v < 1$.}  In this case we have the bound
$$ \frac{v}{p+1} |u|^{p+1} \leq (\T_{01} + v \T_{00}) + O( \frac{R^{1/2}}{T^{1/2}} \T_{00} )$$
and so from \eqref{t2b} and \eqref{energy-cons} we have
$$ \frac{v}{p+1} \int_\R \int_\R \chi(\frac{t}{T}) \chi(\frac{x-vt}{R}) |u(t,x)|^{p+1}\ dt dx \lesssim R  + R^{1/2} T^{1/2} $$
and \eqref{t0t} follows.

{\bf Case 3: (Timelike case) $0 \leq v \leq 1 - \frac{R^{1/2}}{2T^{1/2}}$.}  Here we use the identity
$$ (\T_{11} + v \T_{01}) + v (\T_{01} + v\T_{00}) + \frac{1-v^2}{4} Q
= (v u_t + u_x)^2 + \frac{(p-1)(1-v^2)}{2(p+1)} |u|^{p+1}.$$
Taking the indicated linear combination of \eqref{t1b}, \eqref{t2b}, \eqref{t3b} and discarding the non-negative quantity $(v u_t + u_x)^2$, we conclude that
$$ \frac{(p-1)(1-v^2)}{2(p+1)}
\int_\R \int_\R \chi(\frac{t}{T}) \chi(\frac{x-vt}{R}) |u(t,x)|^{p+1}\ dt dx
\lesssim R + \frac{1-v^2}{4} \frac{T}{R}$$
and thus (noting that $1-v^2 = (1-v)(1+v)$ is comparable to $1-v$)
$$
\int_\R \int_\R \chi(\frac{t}{T}) \chi(\frac{x-vt}{R}) |u(t,x)|^{p+1}\ dt dx
\lesssim \frac{R}{1-v} + \frac{T}{R}.$$
Since $1-v \gtrsim R^{1/2}/T^{1/2}$ by hypothesis, the claim follows.
\end{proof}

\section{Proof of Theorem \ref{main}}

We are now ready to prove Theorem \ref{main}.  Suppose that this claim failed for some $E, p$.  Carefully negating the quantifiers, we may thus find a sequence of times $T_n \to \infty$ and $t_n \in \R$, a $\delta > 0$ independent of $n$, and a family of solutions $u_n$ which uniformly obey the energy bound $E[u_n] \leq E$ such that
$$ \frac{1}{2T_n} \int_{t_n-T_n}^{t_n+T_n} \|u_n(t)\|_{L^\infty_x(\R)}\ dt \geq \delta.$$
By translating each $u_n$ by $t_n$, we may normalise $t_n=0$.

Let $n$ be large.   We will now allow our implied constants in the $\lesssim$ notation to depend on $\delta$, thus
$$ \int_{-T_n}^{T_n} \|u_n(t) \|_{L^\infty(\R)}\ dt \gtrsim T_n.$$
From this bound and \eqref{ubound}, we now conclude that the set
$$ \{ t \in [-T_n, T_n]: \|u_n(t) \|_{L^\infty(\R)} \gtrsim 1 \}$$
has Lebesgue measure $\gtrsim T_n$ (for suitable choices of implied constants).  In particular, we can find a finite set $\Delta_n \subset [-T_n,T_n]$ of times which are $1$-separated and of cardinality
$$ \# \Delta_n \gtrsim T_n$$
such that
\begin{equation}\label{utx0}
 \|u_n(t) \|_{L^\infty(\R)} \gtrsim 1
\end{equation}
for all $t \in \Delta_n$.

For each $t \in \Delta_n$, let $x_n(t) \in \R$ be a point such that $|u_n(t,x_n(t))| \geq \frac{1}{2} \|u_n(t)\|_{L^\infty(\R)}$.  From \eqref{utx0}, one has
\begin{equation}\label{utx}
|u_n(t,x_n(t))| \gtrsim 1
\end{equation}
for all $t \in \Delta_n$.

Let us say that two times $t, t' \in \Delta_n$ are \emph{spacelike} if we have
$$ |x_n(t') - x_n(t)| \geq |t-t'| + 1.$$

There is a limit as to how many spacelike pairs of times can exist:

\begin{lemma}[Finite speed of propagation]  Let $n$ be sufficiently large, and let $t_1,\ldots,t_m \in \Delta_n$ be times which are pairwise spacelike.  Then we have $m = O(1)$.
\end{lemma}

\begin{proof}  Without loss of generality we may assume that $t_1 < \ldots < t_m$.  Consider the spacetime region
$$ \Omega := \R \times \R \backslash \bigcup_{1 \leq j \leq m} \{ (t,x): t \geq t_j; |x-x_n(t_j)| \leq t-t_j+\frac{1}{2} \}.$$
Standard energy estimates reveal that
$$ \int_{x: (t_{j},x) \in \Omega} \T_{00}(t_j,x)\ dx + \int_{x: |x-x_n(t_j)| \leq \frac{1}{2}} \T_{00}(t_j,x)\ dx \leq \int_{x: (t_{j-1},x) \in \Omega} \T_{00}(t_{j-1},x)\ dx$$
for all $1 < j \leq m$, where $\T_{00} = \T_{00,n}$ is the energy density of $u_n$.  Iterating this and then using \eqref{energy-cons}, we conclude that
$$ \sum_{1 < j \leq m} \int_{x: |x-x_n(t_j)| \leq \frac{1}{2}} \T_{00}(t_j,x)\ dx \lesssim 1$$
and in particular that
$$ \sum_{1 < j \leq m} \int_{x: |x-x_n(t_j)| \leq \frac{1}{2}} |u_n(t_j,x)|^{p+1}\ dx \lesssim 1.$$
But from \eqref{utx}, \eqref{uholder} we see that
$$ \int_{x: |x-x_n(t_j)| \leq \frac{1}{2}} |u_n(t_j,x)|^{p+1}\ dx \gtrsim 1.$$
for each $j$, and the claim follows.
\end{proof}

We now use this lemma and some combinatorial arguments to extract a Lipschitz worldline.

\begin{corollary}[Existence of Lipschitz worldline]\label{lipcor} Let $\eps_0: (0,1] \to (0,1]$ be an arbitrary function.  Then there exists a constant $0 < c_0 = c_0(\eps_0) \leq 1$ with the following property: for all sufficiently large $n$, there exists $c_0 < c < 1$ (depending on $n$) and a subset $\Delta'_n$ of $\Delta_n$ with
$$ \# \Delta'_n \geq c T_n$$
such that we have the Lipschitz property
\begin{equation}\label{xtt}
 |x_n(t') - x_n(t)| \leq |t-t'| + \eps_0(c) T_n
\end{equation}
for all $t, t' \in \Delta'_n$.
\end{corollary}

\begin{proof}  Fix $\eps$, and let $n$ be sufficiently large.  Define the \emph{particle number} of a set $\Delta$ to be the largest integer $m$ for which one can find pairwise spacelike times $t_1,\ldots,t_m$ in $\Delta$.  By the previous lemma, we see that $\Delta_n$ has particle number $O(1)$. The key lemma is the following:

\begin{lemma}[Dichotomy]\label{dich}  Let $\Delta' \subset \Delta_n$, $m=O(1)$ and $c > 0$ be such that
$$ \# \Delta' \geq 2c T_n$$
and $\Delta'$ has particle number at most $m$.  Suppose $n$ is sufficiently large depending on $c$. Then at least one of the following is true:
\begin{itemize}
\item[(i)] There exists a subset $\Delta'' \subset \Delta'$ of cardinality at least $cT_n$ such that \eqref{xtt} holds for all $t, t' \in \Delta''$.
\item[(ii)] There exists a subset $\Delta''' \subset \Delta'$ of cardinality at least $c \eps_0(c) T_n/16$ with particle number at most $m-1$.
\end{itemize}
\end{lemma}

Iterating this lemma at most $O(1)$ times we obtain the claim.

It remains to prove the lemma.  We subdivide the interval $[-T_n,T_n]$ into intervals $I$ of length between $\eps_0(c) T_n/4$ and $\eps_0(c) T_n/8$.  Call an interval \emph{sparse} if $\#( \Delta' \cap I ) \leq c \eps_0(c) T_n /8$, and \emph{dense} otherwise.  Observe that at most $c T_n$ elements of $\Delta'$ lie in sparse intervals.  Thus if we let $\Delta''$ denote the intersection of $\Delta'$ with the union of all the dense intervals, then $\# \Delta'' \geq c T_n$.

If $\Delta''$ obeys \eqref{xtt} then we are done.  Otherwise, we can find $t_1, t_2 \in \Delta''$ such that
$$ |x_n(t_1) - x_n(t_2)| > |t_1-t_2| + \eps_0(c) T_n.$$
The time $t_1$ must lie in some dense interval $I$.  We split $\Delta'' \cap I = \Delta'''_1 \cup \Delta'''_2$, where $\Delta'''_1$ consists of all $t \in \Delta'' \cap I$ with $|x_n(t)-x_n(t_1)| \leq \eps_0(c) T_n/2$, and $\Delta'''_2$ consists of the remainder of $\Delta'' \cap I$.  Observe from the triangle inequality (if $n$ is sufficiently large depending on $c$) that all times in $\Delta'''_1$ are spacelike with respect to $t_2$, and similarly all times in $\Delta'''_2$ are spacelike with respect to $t_1$.  Thus each of $\Delta'''_1$ and $\Delta'''_2$ can have particle number at most $m-1$.  On the other hand, by the pigeonhole principle, one of $\Delta'''_1$ and $\Delta'''_2$ must have cardinality at least $\frac{1}{2} \# (\Delta'' \cap I)$, which is at least $c \eps_0(c) T_n / 16$ since $I$ is dense.  The lemma, and hence the corollary, follows.
\end{proof}

Let $\eps_0: (0,1] \to (0,1]$ to be a function to be chosen later (one should think of $\eps_0(c)$ as going to zero very rapidly as $c \to 0$).  For any sufficiently large $n$, let $c_0, c$ and $\Delta'_n$ be as in Corollary \ref{lipcor}.

Define the function $x'_n: [-T_n,T_n] \to \R$ by
$$ x'_n(t) := \inf_{t' \in \Delta'_n} (x_n(t') + |t-t'|).$$
One easily verifies that $x'_n$ is Lipschitz with constant at most $1$. From \eqref{xtt} we also see that
\begin{equation}\label{xt-close}
|x_n(t) - x'_n(t)| \leq \eps_0(c) T_n
\end{equation}
for all $t \in \Delta'_n$.

We now apply a quantitative version of the Rademacher (or Lebesgue) differentiation theorem to ensure that $x'_n(t)$ is approximately differentiable on a large interval.

\begin{proposition}[Quantitative Rademacher differentiation theorem]  Let $\eps_1: (0,1] \to (0,1]$ be a function, and let $\delta > 0$.  Then there exists $r_1 = r_1(\eps_1,\delta) > 0$ with the following property: given any Lipschitz function $f: [-1,1] \to \R$ with Lipschitz constant at most $1$, there exists $r_1 \leq r \leq 1$ such that the set
\begin{align*}
& \{ x \in [-1,1]: \hbox{There exists } L \in \R \hbox{ such that } |\frac{f(y)-f(x)}{y-x}-L| \leq \delta \\
&\quad \hbox{ whenever } y \in [-1,1] \hbox{ is such that } \eps_1(r) \leq |y-x| \leq r \}
\end{align*}
(which, intuitively, is the set where $f$ is approximately differentiable) has Lebesgue measure at least $2-\delta$.
\end{proposition}

\begin{proof}  We give an indirect ``compactness and contradiction'' proof.  Suppose for contradiction that the claim failed.  Negating the quantifiers carefully, this means that there exists a function $\eps_1: (0,1] \to (0,1]$, a $\delta > 0$, a sequence $r_n \to 0$, and a sequence $f_n: [0,1] \to \R$ of Lipschitz functions with constant at most $1$, such that the sets
\begin{align*}
&\{ x \in [-1,1]: \hbox{There exists } L \in \R \hbox{ such that } |\frac{f_n(y)-f_n(x)}{y-x}-L| \leq \delta \\
&\quad \hbox{ whenever } y \in [-1,1] \hbox{ is such that } \eps_1(r) \leq |y-x| \leq r \}
\end{align*}
have Lebesgue measure at most $2-\delta$ for all $n$ and all $r_n \leq r \leq 1$.

By translating each $f_n$ by a constant if necessary, we may normalise $f_n(0)=0$.  The Lipschitz functions then form a bounded equicontinuous family on the compact domain $[-1,1]$, and so by the Arzela-Ascoli theorem we may (after passing to a subsequence if necessary) assume that the $f_n$ converge uniformly to a limit $f$.  We conclude that the set
\begin{align*}
&\{ x \in [-1,1]: \hbox{There exists } L \in \R \hbox{ such that } |\frac{f(y)-f(x)}{y-x}-L| \leq \delta/2 \\
&\quad \hbox{ whenever } y \in [-1,1] \hbox{ is such that } \eps_1(r) \leq |y-x| \leq r \}
\end{align*}
has Lebesgue measure at most $2-\delta$ for all $0 < r \leq 1$.  On the other hand, $f$ is clearly Lipschitz with constant at most $1$, and so by the Lipschitz differentiation theorem, $f$ is differentiable almost everywhere.  In particular, the set
\begin{align*}
&\bigcup_{m=1}^\infty \{ x \in [-1,1]: \hbox{There exists } L \in \R \hbox{ such that } |\frac{f(y)-f(x)}{y-x}-L| \leq \delta/2 \\
&\quad \hbox{ whenever } y \in [-1,1] \hbox{ is such that } 0 < |y-x| \leq 2^{-m} \}
\end{align*}
has full measure in $[-1,1]$.  By the monotone convergence theorem, this implies that one of the sets in this union has measure greater than $2-\delta$.  But this contradicts the previous claim.
\end{proof}

\begin{remark}  It is also possible to give a more direct ``martingale''\footnote{Indeed, the arguments here are closely related to some classical martingale inequalities of Doob\cite{doob} and L\'epingle\cite{lep}.} or ``multiscale analysis'' proof of this proposition, which we sketch as follows.  For each $n \geq 1$, let $f_n$ be the piecewise linear continuous function which agrees with $f$ on multiples of $2^{-n}$, and is linear between such intervals.  One easily verifies that the functions $f_{n+1}-f_n$ are pairwise orthogonal in the Hilbert space $\dot H^1([-1,1])$, and thus by Bessel's inequality we have
$$ \sum_{n=1}^\infty \|f_{n+1}-f_n\|_{\dot H^1([-1,1])}^2 \leq 2.$$
Now let $F: \N \to \N$ be a function to be chosen later, and let $\sigma > 0$ be a small quantity to be chosen later.  From the pigeonhole principle, one can find $1 \leq n_0 \leq C( F, \sigma )$ such that
$$ \sum_{n=n_0}^{F(n_0)} \|f_{n+1}-f_n\|_{\dot H^1([-1,1])}^2 \leq \sigma.$$
If one then sets $r := \sigma 2^{-n_0}$, one can verify all the required claims if $\sigma$ is chosen sufficiently small depending on $\delta$, and $F$ is sufficiently rapidly growing depending on $\delta$, $\sigma$, and $\eps_0$; the quantity $L$ can basically be taken to be $f'_n(x)$.  We omit the details, but see \cite{tao:besi} for some similar arguments in this spirit.
\end{remark}

Let $\delta > 0$ be a small quantity (depending on $c$) to be chosen later, and let $\eps_1: (0,1] \to (0,1]$ be the function $\eps_1(r) := \delta r$. We let $n$ be sufficiently large, and apply the above proposition to the Lipschitz function $f = f_n: [-1,1] \to \R$ defined by $f(y) := \frac{1}{T_n} x'_n(T_ny)$.  We conclude that there exists $r_1 = r_1(\delta)$ and $r_1 < r < 1$ (depending on $\delta$ and $n$) such that the set
\begin{align*}
\{ t \in [-T_n,T_n]: &\hbox{There exists } L \in \R \hbox{ such that } |\frac{x'_n(t')-x'_n(t)}{t'-t}-L| \leq \delta \\
&\quad \hbox{ whenever } t' \in [-T_n,T_n] \hbox{ is such that } \delta r T_n \leq |t-t'| \leq r T_n \}
\end{align*}
has measure at least $(2-\delta) T_n$.

On the other hand, the set $\Delta'_n$ has cardinality at least $cT_n$.  As in the proof of Lemma \ref{dich}, we partition $[-T_n,T_n]$ into intervals $I$ of length between $rT_n/4$ and $rT_n/8$, and let $\Delta''_n$ be the portion of $\Delta'_n$ which are contained inside those intervals $I$ which are \emph{dense} in the sense that they contain at least $cr T_n/16$ elements of $\Delta'_n$.  It is easy to see that $\Delta''_n$ has cardinality at least $cT_n/2$.  Also, $\Delta''_n$ is $1$-separated.

Thus, if we let $\delta = \delta(c)$ be sufficiently small compared to $c$, we can find $t_* \in [-T_n,T_n]$ within a distance $1$ of $\Delta''_n$ and $v \in \R$ such that
\begin{align*}
&\left|\frac{x'_n(t')-x'_n(t_*)}{t'-t_*}-v\right| \leq \delta \hbox{ whenever } t' \in [-T_n,T_n] \\
\quad &\hbox{ is such that } \delta r T_n \leq |t_*-t'| \leq r T_n.
\end{align*}
Let $t_0$ be an element of $\Delta''_n$ within $1$ of $t_*$.  Applying \eqref{xt-close}, the triangle inequality, and the Lipschitz nature of $x'_n$, we conclude that
$$ x_n(t_1) = x_n(t_0) + v (t_1-t_0) + O( \delta |t_1-t_0| ) + O( \eps_0(c) T_n ) + O( 1 )$$
whenever $t_1 \in \Delta''_n$ is such that $ \delta T_n+1 \leq |t_1-t_0| \leq r T_n-1$.  
Applying the Lipschitz property again, we conclude that
$$ x_n(t_1) = x_n(t_0) + v (t_1-t_0) + O( \delta r T_n ) + O( \eps_0(c) T_n ) + O(1)$$
for all $t_1 \in \Delta''_n$ with $|t_1-t_0| \leq rT_n-1$.  
If we set $\eps_0(c) := \delta(c) r_1(\delta(c))$, and assume $n$ is sufficiently large depending on all other parameters, we thus have
$$ x_n(t_1) = x_n(t_0) + v (t_1-t_0) + O( \delta r T_n )$$
whenever $t_1 \in \Delta''_n$ and $|t_1-t_0| \leq rT_n/4$.
One should view this as an assertion that $x_n$ is approximately differentiable near $t_0$.

By definition of $\Delta''_n$, we know that $t_0$ is contained in an interval $I$ of length at most $rT_n/4$ which contains $\gtrsim cr T_n$ elements of $\Delta_n$.  We thus see that the parallelogram
$$ P := \{ (t,x): t \in I; |x - x_n(t_0) - v(t-t_0)| \leq R/2 \} $$
contains at least $\gtrsim cr T_n$ points of the form $(t, x_n(t))$ with $t \in \Delta_n$, where $R$ is a quantity of size $\sim \delta r T_n$.
On the other hand, by definition of $\Delta_n$, we have $|u_n(t,x(t))| \gtrsim 1$ for all $t \in \Delta_n$.  Applying \eqref{uholder}, we conclude that
$$ \int_P |u_n(t,x)|^{p+1}\ dt dx \gtrsim cr T_n.$$
On the other hand, from Proposition \ref{parallel} we have
$$ \int_P |u_n(t,x)|^{p+1}\ dt dx \lesssim R^{1/2} (rT_n)^{1/2} + \frac{rT_n}{R} \lesssim \delta^{1/2} r T_n + \delta^{-1}.$$
If we set $\delta$ to be sufficiently small depending on $c$, and let $n$ be sufficiently large depending on all other parameters, we obtain a contradiction as desired.  This completes the proof of Theorem \ref{main}.

\end{document}